\documentclass[10pt]{amsart}
\input xy
\xyoption{all}

\usepackage{amsthm}
\usepackage{amssymb}
\usepackage{amsmath}
\usepackage{mathtools}
\usepackage{pdfpages}
\usepackage{float}
\usepackage[colorlinks=true,pdftex,unicode=true,linktocpage,bookmarksopen,hypertexnames=false]{hyperref}
\usepackage{tikz-cd}
\usepackage{caption}

\newcommand{\Cal}[1]{{\mathcal #1}}

\DeclareMathOperator{\id}{id}

\DeclareMathOperator{\Spec}{Spec}

\makeatletter
\numberwithin{equation}{section}
\numberwithin{figure}{section}
\numberwithin{table}{section}
\newtheorem{thm}{Theorem}[section]
\newtheorem*{thm*}{Theorem}
\newtheorem{lem}[thm]{Lemma}
\newtheorem{cor}[thm]{Corollary}
\newtheorem{pro}[thm]{Proposition}

\newtheorem{defn}[thm]{Definition}

\newtheorem{rem}[thm]{Remark}
\newtheorem{exa}[thm]{Example}

\makeatother

\title[]{Ideals and congruences in $L$-algebras and pre-$L$-algebras}
  \author[Alberto Facchini]{Alberto Facchini}
\address{Dipartimento di Matematica ``Tullio Levi-Civita'', Universit\`a di Padova, 35121\linebreak Padova, Italy}
 \email{facchini@math.unipd.it}
\thanks{The first author wishes to express his gratitude to the Institut de Recherche en Ma\-th\'{e}\-ma\-tique et Physique of the Universit\'{e} catholique de Louvain for the hospitality received during his stay in Louvain-la-Neuve. }

 \author[Marino Gran]{Marino Gran}
\address{ Institut de Recherche en Math\'{e}matique et Physique, Universit\'{e} catholique de Louvain, 1348 Louvain-la-Neuve, Belgique}
 \email{marino.gran@uclouvain.be}
\thanks{The second author's work was supported by the Fonds de la Recherche Scientifique - FNRS under Grant CDR no. J.0080.23 }

  \author[Mara Pompili]{Mara Pompili}
\address{University of Graz, Institute for Mathematics and Scientific Computing, Heinrichstrasse 36, 8010 Graz, Austria}
 \email{mara.pompili@uni-graz.at}

\subjclass[2020]{}

\begin{document}

\begin{abstract}
We link the recent theory of $L$-algebras to previous notions of Universal Algebra and Categorical Algebra concerning subtractive varieties,  commutators, multiplicative lattices, and their spectra. 
We show that the category of $L$-algebras is subtractive and normal in the sense of Zurab Janelidze, but neither the category of $L$-algebras nor that of pre-$L$-algebras are Mal'tsev categories, hence in particular they are not semi-abelian. Therefore $L$-algebras are a rather peculiar example of an algebraic structure.
\end{abstract}
 
\maketitle

\section{Introduction}

The aim of this paper is to link the recent fruitful theory of $L$-algebras \cite{RumpL, RumpJAA,
RV} to previous articles of Universal Algebra and Categorical Algebra concerning subtractive varieties \cite{Ursini, 
Ursini2, Zurab-Sub, Zurab},  commutators and multiplicative lattices, and their spectra \cite{FFJ}. 

$L$-algebras are related to right $\ell$-groups, projection lattices of von Neumann algebras, quantum Yang–Baxter equation, MV-algebras, braidings, and non-commuta\-tive logic.
In this paper we prove that the category of $L$-algebras is subtractive \cite{Ursini, Zurab-Sub} and normal (in the sense of \cite{Zurab}), but neither the category of $L$-algebras nor that of pre-$L$-algebras are Mal'tsev categories, hence in particular they are not semi-abelian. This shows that $L$-algebras are a rather peculiar example of algebraic structure. In general, in a pre-$L$-algebra~$X$, there is a monotone Galois
connection between the lattice $\Cal C(X)$ of congruences of $X$, and
the lattice $\Cal I(X)$ of ideals of $X$. We show that $L$-algebras do not form a variety in the sense of Universal Algebra, but their category is a normal subtractive quasivariety. The lattice $\Cal I(X)$ is a distributive lattice, as was proved in \cite{RV}. Thus the commutator of two ideals (congruences) turns out to be the intersection of the two ideals (congruences).

\medskip

We are grateful to Professor Leandro Vendramin for several suggestions concerning this paper.

\section{Basic notions}

Recall that an {\em $L$-algebra} \cite{RumpL} is a set $X$ with a binary operation $(x, y) \mapsto x \cdot y$ and a $0$-ary operation $1 \in X $ such that 
\begin {align} 
&x \cdot x = x \cdot 1 = 1, \ 1 \cdot x = x,\label{1}\\
&(x \cdot y) \cdot(x \cdot z) = (y \cdot x) \cdot (y \cdot z), \ \mbox{\rm and}\label{2}\\
&x \cdot y = y \cdot x = 1 \implies x = y \label{3}
\end{align} 
for every $x, y, z \in X$.

Similarly, we define a {\em pre-$L$-algebra} assuming that only Properties~\eqref{1} and \eqref{2} hold. These pre-$L$-algebras are called {\em unital cycloids} in \cite{RumpL}. It is easily seen that in a pre-$L$-algebra the element $1$ with Property~\eqref{1} is unique (this follows from $x\cdot x=1$.) It is called the {\em logical unit} of the pre-$L$-algebra.
Equation (\ref{2}) holds in most 
generalizations of classical logic, including intuitionistic, many-valued, and
quantum logic.

On any pre-$L$-algebra $X$ there is a natural preorder $\le$, i.e., a reflexive and transitive relation, defined by $x\le y$ if $x\cdot y=1$. A pre-$L$-algebra is an $L$-algebra if and only if this natural preorder is a partial order, that is, if and only if it is also an antisymmetric relation. Clearly, pre-$L$-algebras form a variety of algebras in the sense of Universal Algebra. (We will see in Example~\ref{4} that this is not the case for $L$-algebras.) In particular, a pre-$L$-algebra morphism is any mapping $f$ between two pre-$L$-algebras such that $f(1)=1$ and $f(x\cdot y)=f(x)\cdot f(y)$ for all $x,y$. But notice that the first condition $f(1)=1$ follows from the second, because 
\[
f(1)=f(1\cdot 1)=f(1)\cdot f(1)=1
\]
by \eqref{1}.

Let us now consider congruences on a pre-$L$-algebra $X$. We will see that they correspond to suitably defined ideals of $X$. A subset $I$ of a pre-$L$-algebra $X$ is an {\em ideal} of $X$ \cite[Definition 1]{RumpL} if \begin{align}
&1 \in I,\\ 
&x \in I \text{ and }x \cdot y \in I \implies y \in I,\\
&x \in I \implies (x \cdot y) \cdot y \in I,\\
&x \in I \implies y \cdot x \in I,\\
&x \in I \implies y \cdot (x \cdot y) \in I
\end{align}
for every $x,y\in X$. Clearly, $\{1\}$ and $X$ are ideals in any pre-$L$-algebra $X$.

For instance, it is easily seen that if $f\colon X\to Y$ is a pre-$L$-algebra morphism, then its {\em kernel}, that is, the inverse image $f^{-1}(1)$ of the logical unit $1$ of $Y$, is an ideal of $X$.

Also, if $\sim$ is a congruence on a pre-$L$-algebra $X$, the equivalence class $[1]_{\sim}$ of the logical unit $1$ of $X$ is an ideal of $X$. 

\medskip

Recall that if $(A,\le)$ and $(B,\le)$ are two partially ordered sets, a {\em monotone Galois connection} between $A$ and $B$ consists of two order-preserving mappings $f\colon A\to B$ and $g\colon B\to A$  such that $f(a)\le b$ if and only if $a\le g(b)$ for every $a\in A$ and $b\in B$. A {\em closure operator} on the partially ordered set $A$ is a mapping $c\colon A\to A$ such that $x\le c(y)$ if and only if $c(x)\le c(y)$ for every $x,y\in A$.

\begin{pro} \label{vho}
Let $X$ be a pre-$L$-algebra, $\mathcal{C}(X)$ its lattice of congruences, and $\mathcal{I}(X)$ the lattice of ideals of $X$. Define two mappings $$\varphi\colon\mathcal{C}(X)\to\mathcal{I}(X), \qquad\varphi\colon\sim{}\!\in\mathcal{C}(X)\mapsto [1]_{\sim},$$ that maps any congruence $\sim$ of $X$ to the equivalence class modulo $\sim$ of the logical unit $1$, and $$\psi\colon\mathcal{I}(X)\to\mathcal{C}(X), \qquad \varphi\colon I\in\mathcal{I}(X)\mapsto {\sim}_I,$$ where ${\sim}_I$ is the congruence of $X$ defined, for every $x,y\in X$, by $x\sim_I y$ if both $x\cdot y$ and $y\cdot x$ belong to $I$. Then:
\begin{itemize}
\item[(a)] $\varphi$ and $\psi$ are well-defined.
\item[(b)] $\varphi\psi=\id_{\mathcal{I}(X)}$;
\item[(c)] $\varphi$ and $\psi$ form a monotone Galois connection.
\item[(d)] $\psi\varphi$ is a closure operator on $\mathcal{C}(X)$.
\item[(e)] The image of $\psi$ is the set of all congruences $\sim$ of $X$ for wchich $X/{\sim}$ is an $L$-algebra.
\item[(f)] For every $\sim\,\in\mathcal{C}(X)$, $\psi\varphi(\sim)$ is the smallest congruence $\equiv$ on $X$ that contains $\sim$ and is such that $X/{\equiv}$ is an $L$-algebra.
\end{itemize}
\end{pro}

\begin{proof} (a) This is proved in \cite[Proposition 1]{RumpL}. Notice that here ``$\varphi$ is a well-defined mapping'' means that $\varphi(\sim):= [1]_{\sim}$ is an ideal of $I$ for every congruence $\sim$ on $X$. As far as ``$\psi$ is well-defined'' is concerned, we mean that, for every ideal $I$ of $X$, the relation $\sim_I$ on $X$, defined, for every $x,y\in X$, by $x\sim_I y$ if $x\cdot y\in I$ and $y\cdot x\in I$, is a congruence on $X$. 

(b) is trivial.

(c) $\varphi$ and $\psi$ are clearly order-preserving. In order to show that they form a monotone Galois connection, we must prove that if $\sim $ is any congruence and $I$ is any ideal, then $[1]_{\sim}\subseteq I$ if and only if $\sim{}\!\subseteq{}\!\sim_I$. Suppose $[1]_{\sim}\subseteq I$. Fix $x,y\in X$ with $x\sim y$. Then $x\cdot y\sim y\cdot y=1\in I$. Similarly $y\cdot x=1\in I$. Therefore $x\sim_I y$. The converse is easy.

(d) follows immediately from (c).

(e) Clearly, if $I$ is an ideal of $X$, then $X/{\sim_I}$ is an $L$-algebra. Conversely, let $\sim$ be a congruence with $L/\sim$ an $L$-algebra, and $I$ be the ideal $\varphi(\sim)$. We must prove that $\sim{}\!={}\!\sim_I$, i.e., that for every $x,y\in X$, $x\sim y$ if and only if $x\cdot y\in I$ and $y\cdot x\in I$. Now $x\sim y$ implies $x\cdot y=y\cdot y=1$, so $x\cdot y\in I$. Similarly, $y\cdot x\in I$. Conversely, $x\cdot y\in I$ and $y\cdot x\in I$ are equivalent to $x\cdot y\sim 1$ and $y\cdot x\sim 1$. But $L/\sim$ is an $L$-algebra, hence we have that $[x]_{\sim}=[y]_{\sim}$, that is, $x\sim y$, as desired.

(f) From (d) and (e) we know that $\psi\varphi(\sim)$ is a congruence on $X$ that contains $\sim$ and that $X/{\psi\varphi(\sim)}$ is an $L$-algebra. If $\equiv$ is any other congruence on $X$ that contains $\sim$ and is such that $X/{\equiv}$ is an $L$-algebra, then $\equiv{}\!\!=\!\!{}\sim_I$ for some ideal $I$ by (e), so that $\sim_I\supseteq\sim$. It follows that $\varphi(\sim_I)\supseteq\varphi(\sim)$, that is, $I\supseteq\varphi(\sim)$. Therefore $\equiv{}={}\sim_I{}={}\psi(I)\supseteq\psi\varphi(\sim)$, as desired.\end{proof}

As a trivial consequence of the previous proposition, we have that for any pre-$L$-algebra $X$ there is a one-to-one correspondence between ideals of $X$ and congruences $\sim$ on $X$ for which $X/{\sim}$ is an $L$-algebra. Cf.~\cite[Corollary~1]{RumpL}.

\bigskip

Clearly, the category of all $L$-algebras is a full reflective subcategory of the category of all pre-$L$-algebras. The left adjoint of the inclusion associates with any pre-$L$-algebra $X$ the $L$-algebra $X/{\sim}$, where $\sim$ is the congruence on $X$ defined, for every $x,y\in X$, by $x\sim y$ if $x\cdot y=1$ and $y\cdot x=1$.

\bigskip

Also notice that for every pre-$L$-algebra morphism between two $L$-algebras, the kernel pair always corresponds to an ideal of the domain. This occurs because any pre-$L$-subalgebra of an $L$-algebra is an $L$-algebra.

\section{L-algebras do not form a variety}\label{4}

\begin{exa}{\rm 
    Let $X=\{x,y,z,1\}$ be the $L$-algebra given by
    Table \ref{tab:L} and 
    let $Y=\{a,b,1\}$ be the magma
    given by Table \ref{tab:notL}. 

    \begin{table}[ht]
    \begin{tabular}{|c|cccc|}
        \hline 
         & $x$ & $y$ & $z$ & $1$\tabularnewline
        \hline 
        $x$ & $1$ & $y$ & $z$ & $1$\tabularnewline
        $y$ & $1$ & $1$ & $x$ & $1$\tabularnewline
        $z$ & $1$ & $x$ & $1$ & $1$\tabularnewline
        $1$ & $x$ & $y$ & $z$ & $1$\tabularnewline
        \hline 
        \end{tabular}
        \caption{An $L$-algebra.}
        \label{tab:L}
    \end{table}
    
    \begin{table}[ht]
    \begin{tabular}{|c|ccc|}
    \hline 
     & $a$ & $b$ & $1$\tabularnewline
    \hline 
    $a$ & $1$ & $1$ & $1$\tabularnewline
    $b$ & $1$ & $1$ & $1$\tabularnewline
    $1$ & $a$ & $b$ & $1$\tabularnewline 
    \hline
    \end{tabular}
    \caption{A pre-$L$-algebra that is not an $L$-algebra.}
    \label{tab:notL}
    \end{table}
    
    One checks that $Y$ is not an $L$-algebra, as 
    $a\cdot b=b\cdot a=1$ but $a\ne b$. 
    The surjective map $f\colon X\to Y$, $f(x)=f(1)=1$, $f(y)=a$ and $f(z)=b$, satisfies
    $f(u\cdot v)=f(u)\cdot f(v)$ for all $u,v\in X$ and $f(1)=1$. 
    Since the image of $f$ is not an $L$-algebra, 
    it follows from Birkhoff's theorem (see for example \cite[Theorem 3.1]{Cohn})
    that the class of $L$-algebras is not a variety. }
\end{exa}

\section{Subtractive and normal categories}
Let us recall some definitions that are of interest in categorical algebra. A variety of universal algebras $\mathbb V$ is {\em subtractive} \cite{Ursini} if its algebraic theory contains a constant $0$ and a binary term $s(x,y)$ with the properties that $s(x,x)= 0$ and $s(x,0)= x$.

\begin{lem}\label{Lemma }
The variety $\mathsf{PreLAlg}$ of pre-$L$-algebras is a subtractive variety.
 \end{lem}
 \begin{proof}
It suffices to choose the term $s(x,y) = y \cdot x$ in the theory of pre-$L$-algebras: this term is such that $$s(x,x) = 1$$ and $$s(x, 1)= 1 \cdot x = x.$$
 \end{proof}

The definition of subtractive variety was extended to a categorical context by Z. Janelidze in \cite{Zurab-Sub}. When a category $\mathcal C$ is \emph{pointed}, that is it has a zero object $0$, the property of subtractivity can be defined as follows.
In a pointed category consider a reflexive relation $(R,r_1,r_2, e)$ on an object $X$, where $r_1$, $r_2$ are the projections and $e \colon X \rightarrow R$ is the morphism giving the reflexivity: $r_1 \circ  e = 1_X = r_2 \circ e$. One says that the relation $(R,r_1,r_2, e)$ is {\em right} (resp. {\em left}) {\em punctual} if there is a morphism $t \colon X \rightarrow R$ (resp. $s \colon X \rightarrow R$) such that $r_2 \circ t = 1_X$ and $r_1 \circ t =0$ (resp. $r_2 \circ s = 0$ and $r_1 \circ s =1_X$). 
\begin{defn}\label{subtractive-def}\cite{Zurab-Sub}
A finitely complete pointed category is \emph{subtractive} if any right punctual reflexive relation is left punctual.
\end{defn}
As shown in \cite{Zurab-Sub}, a pointed variety is subtractive in the sense of Definition \ref{subtractive-def} if and only if is subtractive in the sense of \cite{Ursini}. In particular the variety $\mathsf{PreLAlg}$ is a subtractive category, and this implies the following:
\begin{lem}\label{LAlg-Sub}
     The category $\mathsf{LAlg}$ is subtractive.
\end{lem}
\begin{proof}
This follows from the fact that the inclusion functor $U \colon \mathsf{LAlg} \rightarrow \mathsf{PreLAlg}$ is full, faithful, preserves all limits (as any right adjoint), hence it preserves and reflects punctual right (resp. left) reflexive relations, and the zero object. 
\end{proof}

The next result we shall prove concerns $L$-algebras again. They turn out to form a normal category in the following sense:
\begin{defn}\label{normal} \cite{Zurab} 
Let $\mathcal C$ be a finitely complete category with a zero object $0$. Then $\mathcal C$ is a \emph{normal} category if
\begin{enumerate}
\item any morphism $f \colon A \rightarrow B$ admits a factorization $f= m \cdot q$, where $q$ is a normal epimorphism (=a cokernel) and $m$ is a monomorphism:
$$
\xymatrix{A \ar[rr]^f \ar@{->>}[dr]_q & & B \\
&{C\, \, } \ar@{>->}[ur]_m &
}
$$
\item normal epimorphisms are stable under pullbacks, i.e., given any pullback  
$$
\xymatrix{{A \times_ C B} \ar@{->>}[r]^-{p_2} \ar[d]_{p_1} &B \ar[d]^g  \\A  \ar@{->>}[r]_f &  C}
$$
in~$\mathcal C$, where $f$ is a normal epimorphism, then $p_2$ is a normal epimorphism.
\end{enumerate}
\end{defn}
Equivalently, one can define a normal category as a regular category \cite{Barr} with a zero object with the property that any regular epimorphism is a cokernel. 

Observe then that $\mathsf{LAlg}$ is a quasivariety, since it is the class of algebras determined by adding a finite number of implications to the theory of $\mathsf{PreLAlg}$.
Example~\ref{4} shows that these implications cannot be ``transformed" into an equivalent set of identities, since in that case $\mathsf{LAlg}$ would then be a subvariety of $\mathsf{PreLAlg}$ and then it would be stable in it under quotients, and this is not the case. As a consequence $\mathsf{LAlg}$ is not a variety of algebras.

\begin{pro}\label{Proposition}
The category $\mathsf{LAlg}$ of $L$-algebras is a normal subtractive category.
\end{pro}
\begin{proof}
 Let us check that in $\mathsf{LAlg}$ every surjective homomorphism is a cokernel. The terms $t_1(x,y) = x \cdot y$ and $t_2(x,y) = y \cdot x$ have the property that $$ (\forall i \in \{ 1,2\}  \,\,  t_i(a,b)  =  1) \Leftrightarrow  a = b. $$
 These terms will attest the $1$-regularity (in the sense of Beutler \cite{Be}) of $\mathsf{LAlg}$ (also see \cite[Theorem 2.2]{JMU}).
 
 Note that the terms $t_1(x,y) = x \cdot y$ and $t_2(x,y) = y \cdot x$ also satisfy the identities $t_i(x,x)=1$, for $i \in \{1, 2\}$. Let $f \colon A \rightarrow B$ be a surjective homomorphism, $\kappa \colon K \rightarrow A$ be the inclusion of its kernel $K$ in $A$ (that we shall omit to simplify the notations), and $g \colon A \rightarrow C$ be any homomorphism such that $(g \circ \kappa) (k) = g(k) = 1$, for every $k \in K$. Now, for any $b \in B$ there is an $a \in A$ such that $f(a)=b$. Let us prove that, by setting $\phi (b) = g(a)$, we get a well-defined function. To this end, we will show that if $a$ and $a'$ are such that $f(a)= f(a')$ one always has that $g(a)=g(a')$. First observe that 
$$f( t_i (a,a') ) = t_i (f(a) , f(a')) =  t_i (f(a) , f(a))  = 1,$$
hence $t_i (a,a') \in K$. This implies that 
 $$t_i (g(a), g(a'))= g(t_i (a, a'))=1,$$
 so that $g(a)= g(a')$, and $\phi$ is well-defined. It remains to prove that $\phi$ is a homomorphism: this is easy, because $\phi \circ f =g$ is a homomorphism by assumption and $f$ is a surjective homomorphism. The uniqueness of the factorization $\phi$ is clear, and $f$ is then necessarily the cokernel of its kernel $\kappa$. 
 
 To conclude that $\mathsf{LAlg}$ is a normal category, now it suffices to observe that any quasivariety is a regular category (see \cite{PV}, for instance) since surjective homomorphisms are always stable under pullbacks in a quasivariety. From Lemma \ref{LAlg-Sub} it follows that $\mathsf{LAlg}$ is indeed a normal subtractive category.
\end{proof}
The subtractive variety $\mathsf{PreLAlg}$ of pre-$L$-algebras then contains the normal subtractive quasivariety $\mathsf{LAlg}$ of $L$-algebras.
One might then wonder whether $\mathsf{LAlg}$ is also a Mal'tsev category \cite{CLP}. The answer to this question is negative, as we are now going explain. For this, let us fix some notations.
The {\em kernel pair} $(Eq(f), p_1, p_2)$ of a morphism $f\colon A \to B$ is the (effective) equivalence relation obtained by  the pullback of $f$ along itself:
$$
\xymatrix{Eq(f) \ar@{->>}[r]^-{p_2} \ar[d]_{p_1} &A \ar[d]^f  \\A  \ar@{->>}[r]_f &  B.}
$$
In the case of a quasivariety of universal algebras, then the kernel pair of a homomorphism $f \colon A \rightarrow B$ is the congruence $$Eq(f) = \{ (a,a') \in A \times A \, \mid \, f(a) = f(a')\}.$$

We'll be interested in the following two elements $L$-algebra $X$ whose multiplication is defined in the following Table:
\begin{table}[ht]
    \begin{tabular}{|c|cc|}
    \hline 
     & $0$ & $1$ \tabularnewline
    \hline 
    $0$ & $1$ &  $1$ \tabularnewline
    $1$ & $0$ & $1$ \tabularnewline
    \hline
    \end{tabular}
    \label{tab:reflexive}
    \end{table}
    
    \noindent Consider then the relation $R = \{(0,1), (1,0), (1,1) \}$ on $X$: this is easily seen to be a subalgebra of the product $L$-algebra $X \times X$.
    
Write $p_1 \colon R \rightarrow X$ and $p_2 \colon R \rightarrow X$ for the first and the second projections, and $Eq(p_1)$ and $Eq(p_2)$ for the congruences associated with these homomorphisms, namely $$Eq(p_1) = \{((0,1),(0,1)), ((1,0),(1,0)), ((1,1),(1,1)), ((1,0),(1,1)) \} $$
and
$$Eq(p_2) = \{((0,1),(0,1)), ((1,0),(1,0)), ((1,1),(1,1)), ((0,1),(1,1)) \}. $$
Then, clearly, 
$$(1,0) Eq(p_1) (1,1) Eq(p_2) (0,1)$$
showing that $$((1,0), (0,1)) \in Eq(p_2) \circ Eq(p_1).$$ However, $$((1,0), (0,1)) \not\in Eq(p_1) \circ Eq(p_2),$$ hence 
$$Eq(p_1) \circ Eq(p_2)\not=  Eq(p_2) \circ  Eq(p_1).$$
This shows that:

\begin{pro}
    
\label{NotMal'tsev}
The categories $\mathsf{LAlg}$ and $\mathsf{PreLAlg}$ are not Mal'tsev categories. 
\end{pro}
\begin{proof}
   The fact that $\mathsf{LAlg}$ is not a Mal'tsev category follows from the fact that the two congruences $Eq(p_1)$ and $Eq(p_2)$ on the $L$-algebra $X$ above do not permute in the sense of composition of relations. Since the category $\mathsf{LAlg}$ is stable in $\mathsf{PreLAlg}$ under subalgebras and products in $\mathsf{LAlg}$, the same counter-example also shows that $\mathsf{PreLAlg}$ is not a Mal'tsev category. \end{proof}
In particular, the above proposition implies that $\mathsf{LAlg}$ and $\mathsf{PreLAlg}$ are not semi-abelian categories.
\begin{rem}{\rm 
It is well-known that any subtractive variety $\mathbb V$ (with a constant $1$) is ``permutable at $1$'', which means that, for any pair of congruences $R$ and $S$ on any  algebra $X$ in $\mathbb V$, the following implication holds:
$$(x,1) \in S \circ R \quad \Leftrightarrow  \quad (x,1) \in R \circ S$$ 

In the case of the variety $\mathsf{PreLAlg}$ we have the term $s(x,y) = y \cdot x$. Accordingly, when  $(x,1) \in S \circ R$, from the existence of a $y$ such that $x R y S 1$, one deduces that $$x = s(x,1) S s(x,y) R s(y,y)= 1,$$ that is, $(x,1) \in R \circ S$. The variety $\mathsf{PreLAlg}$ is then ``permutable at $1$'', even though it is not congruence permutable.}
\end{rem}

\section{Commutators}

In Proposition~\ref{vho} we have considered, for a pre-$L$-algebra $X$, the lattice of congruences $\mathcal{C}(X)$, and the lattice of ideals $\mathcal{I}(X)$. These are complete lattices because any intersection of congruences (of ideals) is a congruence (an ideal). Let us focus onto the case of ideals. Notice that ideals of $X$ are pre-$L$-subalgebras of $X$ (essentially because, for a congruence $\sim$, $x\sim 1$ and $y\sim 1$ imply $x\cdot y\sim 1$). Since the lattice $\mathcal{I}(X)$ is complete, there is an obvious notion of ideal of $X$ generated by a subset of~$X$.

Let us consider the notion of commutator of two ideals of an $L$-algebra $X$. 
In Group Theory, if we have two normal subgroups $M$ and $N$ of a group $G$, the commutator $[M,N]$ is the smallest normal subgroup of $G$ for which group multiplication $\mu\colon M\times N\to G/[M,N]$, $\mu(m,n)=mn[M,N]$, where $mn[M,N]$ is the coset of $mn$ in the quotient $G/[M,N]$, is a group homomorphism. This argument can be repeated for $L$-algebras, as follows.

Let $X$ be an $L$-algebra and $I,J$ be two ideals of $X$. Define their {\em commutator} $[I,J]$ as the smallest ideal of $X$ for which the  multiplication $\cdot$ in $X$, i.e., the mapping $\mu\colon I\times J\to X/[I,J]$, $\mu(i,j)=[i\cdot j]_{\sim_{[I,J]}}$, is an $L$-algebra morphism. 
Notice that the ideal $[I,J]$ is always contained in $I\cap J$. This follows from the remark that
the mapping $\mu\colon I\times J\to X/I\cap J$ is clearly an $L$-algebra morphism. 
One actually has the following



\begin{pro}\label{zjrt} For every pair $I,J$ of ideals of an $L$-algebra $X$, one has $$[I,J]=I\cap J.$$\end{pro}

\begin{proof}
We only need to prove that $I\cap J \subseteq [I,J]$. For this it will suffice to show that, for any $x \in I \cap J$, its equivalence class $[x]_{\sim_{I\cap J}}$, that will be simply written $[x]$, is the neutral element in the quotient $X/[I,J]$:
$$[x]= [1].$$
By assumption, for any $i\in I$, $j \in J$,  one has the equality 
$$([x]\cdot [x])\cdot ([i] \cdot [j]) = ([x] \cdot [i]) \cdot ([x] \cdot [j]).$$
By choosing $i=1$ and $j=x$ we get
$$([x]\cdot [x]) \cdot ([1] \cdot [x]) = ([x] \cdot [1])\cdot ([x] \cdot [x]),$$
from which it follows that
$[x] = [1],$ as desired.
\end{proof}

In particular, this result implies that the only abelian algebras in the quasivariety of $L$-algebras are the trivial ones:

\begin{cor}\label{good}
    Let $X$ be an abelian $L$-algebra. Then $|X|=1$. 
\end{cor}

\begin{proof}
By Proposition \ref{zjrt} the condition 
$[X,X] = \{1\}$ gives $X \cap X = X = \{1\}.$
\end{proof}


\begin{rem}
\emph{The fact that the commutator $[I,J]$ of two ideals is simply their intersection $I \cap J$ is not surprising, since the category $\mathsf{LAlg}$ is congruence distributive, as it follows from Proposition \ref{vho} and the fact that the lattice of ideals on each $L$-algebra is distributive \cite{RV}. It would then be interesting to revisit the results on commutators of congruences in terms of pseudogroupoids in varieties \cite{JP2} in the more general context of quasivarieties (also see \cite{KMK}).}
\end{rem}

We now consider the multiplicative lattice $(\mathcal{I}(X),\cap)$ in the sense of \cite{FFJ}, which has been implicitly considered in \cite{RV}. Now, primes ideals studied in \cite{RV} agree with the notion of prime elements of \cite{FFJ}. Notice that an ideal $P$ of an $L$-algebra $X$ is prime if and only if $P$ is a $\wedge$-irreducible element of the lattice $\mathcal{I}(X)$.

Recall that an ideal $I$ of $X$ is {\em semiprime} if, for every ideal $J$, $[J,J]\subseteq I$ implies $J\subseteq I$. Hence Proposition \ref{zjrt} trivially implies that:

\begin{cor} In an $L$-algebra, every ideal is semiprime.\end{cor}

Thus, in an $L$-algebra, every ideal is an intersection of prime ideals \cite{FFJ}. Now, in any multiplicative lattice, the lattice of all semiprime elements is isomorphic to the lattice of all open subsets of the Zariski spectrum. The Zariski spectrum is always a sober space \cite{FFJ,RV}. Hence we find that:

\begin{pro}{\rm \cite{RV}} For an $L$-algebra $X$, the lattice $\mathcal{I}(X)$ of all ideals of $X$ is isomorphic to the lattice of all open subsets of the sober topological space $\Spec(X)$. In particular, the lattice $\mathcal{I}(X)$ is a complete distributive lattice.\end{pro}

In view of Proposition~\ref{zjrt}, several classical notions of Algebra trivialize for $L$-algebras.
For instance, solvable $L$-algebras, nilpotent $L$-algebras, $L$-algebras with empty Zariski spectrum are only those with one element, the centralizer of any nontrivial ideal is the trivial ideal, the center of any $L$-algebra is the trivial ideal, and the central series and the derived series are always stationary. 

Notice that maximal ideals are prime, because the zero element is $\wedge$-irreducible in the lattice of two elements. It would be interesting to describe simple $L$-algebras.

\end{document}